\documentclass[a5paper, 10pt]{article}

\usepackage{cite, amsmath, amssymb, booktabs, lastpage, graphicx}
\usepackage[hidelinks]{hyperref}

\usepackage{geometry}
\usepackage{setspace}

\usepackage{amsthm}

\theoremstyle{plain}

\newtheorem{theorem}{Theorem}

\theoremstyle{remark}

\theoremstyle{definition}

\newtheorem{example}{Example}

\usepackage{graphicx}
\makeatletter
\renewcommand{\maketitle}{
	\begin{center}

		\baselineskip=0.30in
		{\Large\bfseries \@title} \par
		\vspace{10mm}
		\baselineskip=0.2in
		{\large\bfseries \@author}\par
		\vspace{10mm}
		{\it \@address} \par
		{\small\tt \@email} \par
		\vspace{10mm}
		
	\end{center}
	\vspace{3mm}
}
\newcommand{\address}[1]{\def\@address{#1}}
\newcommand{\email}[1]{\def\@email{#1}}

\makeatother

\usepackage{fancyhdr}

\usepackage[margin=1cm,%
			font=footnotesize,%
			format=hang,%
			labelsep=period,%
			labelfont=bf]{caption}
\allowdisplaybreaks
\title{Some New Results on Energy of Graphs with Self Loops}
\author{Kalpesh M. Popat$^{a}$, Kunal R. Shingala$^b$}

\address{$^a$Atmiya University, Rajkot-360005, Gujarat, India\\
	$^b$Atmiya University, Rajkot-360005, Gujarat, India\\
    }

\email{kalpeshmpopat@gmail.com, shingalakunal999@gmail.com}

\date{\today}

\begin{document}

\maketitle

\begin{abstract}
The graph $G_\sigma$ is obtained from graph $G$ by attaching self loops on $\sigma$ vertices. The energy $ E(G_\sigma)$ of the graph $G_\sigma$ with order $n$ and eigenvalues $\lambda_1,\lambda_2,\dots,\lambda_n$ is defined as $ E(G_\sigma)= \displaystyle \sum_{i=1}^n\left|\lambda_i-\dfrac{\sigma}{n}\right| $. It has been proved that if $\sigma=0\; or\; n$ then $ E(G)=E(G_\sigma) $. The obvious question arise: Are there any graph such that $E(G)=E(G_\sigma)$ and 0$<\sigma<n$? We have found an affirmative answer of this question and contributed a graph family which satisfies this property.
\end{abstract}

\onehalfspacing

\section{Introduction}
For standard terminology and notations related to graph theory, we follow Balakrishnan and Ranganathan \cite{b1} while any terms related to algebra we depend on Lang \cite{b2}.
\par An undirected graph without multiple edges and self-loops is called a simple graph. The adjacency matrix $ A(G) $ of a simple graph $ G $ with vertex set $ \{v_1, v_2, \dots, v_n\} $ is $ n $-ordered symmetric matrix $ A(G)=[a_{ij}] $ such that,
$$a_{ij}=\left\{  \begin{array}{ll}
1&\;\;\; $if\;the\;vertex\;$v_i$\;is\;adjacent\;with\;vertex\;$v_j$,$\\[2mm]
0&\;\;\;$if\;the\;vertex\;$v_i$\;is\;not\;adjacent\;with\;vertex\;$v_j$.$
\end{array}
\right.
$$
\par The characteristic polynomial of the adjacency matrix $ A(G) $ is denoted by $\phi(G:x)$. The roots of characteristic polynomial $\lambda_1, \lambda_2, \dots, \lambda_n$ are called the eigenvalues of graph $G$. The energy $E(G)$ of graph $G$ is developed by Gutman \cite{1} in 1978 as $E(G)=\displaystyle \sum_{i=1}^n\left|\lambda_i\right| $.
\par This graph energy is an emerging subject for a researchers of applied mathematics and mathematical chemistry. A brief account of graph energy of simple graphs can be found in \cite{8,4,9} as well as in the books \cite{b4,b3}. The variants of graph energy can be found in \cite{6,5,7}.
\par Recently the concept of energy of graphs with self-loops is open-up by Gutman \textit{et al.} \cite{2}. Let $G_\sigma$ be the graph obtained from graph $G$ by attaching self loops on $\sigma$ chosen vertices. The adjacency matrix $ A(G_\sigma) $ of graph $G_\sigma$ is an $n \times n$ symmetric matrix such that $A(G_\sigma)=A(G)+I_\sigma$, where $I_\sigma$ is a square matrix of order $n$ with exactly $\sigma$ ones on the main diagonal and all other entries are zero. The eigenvalues of $A(G_\sigma)$ are denoted by $\lambda_1(G_\sigma),\lambda_2(G_\sigma), \cdots, \lambda_n(G_\sigma)$. The energy $E(G_\sigma)$
 of $G_\sigma$ is  
$$E(G_\sigma)=\sum_{i=1}^{n}\left|\lambda_i(G_\sigma)-\frac{\sigma}{n}\right|$$
\par Gutman \textit{et al.} have \cite{2} conjectured that for any graph $G$ of order $n$, $E(G)<E(G_\sigma)$. Irena \textit{et al.} \cite{3} have disproved this conjuncture by showing examples of graphs such that $E(G)>E(G_\sigma)$. It has been shown that \cite{2} if $\sigma=0\;or\;n$ then $E(G)=E(G_\sigma)$. In the present paper we have obtained a graph family such that $E(G)=E(G_\sigma)$ and $0<\sigma<n$.

\section{Main Results}
\begin{theorem}
Let $G$ be the simple graph of order $n$ with eigenvalues $\lambda_1,\lambda_2,\cdots \\,\lambda_n$ and $G^l$ be the graph obtained from $G$ by adding a loop on each vertex of $G$ then $E((G \cup G^l)_n)=2E(G)$, if $|\lambda_i| \ge \frac{1}{2}$, for each $i=1,2,\cdots,n$. 
\end{theorem}
\noindent \textbf{\emph{Proof}}: Let $H_n=G \cup G^l$. The graph $H_n$ contains $2n$ vertices and $n$ loops.  The adjacency matrix of $H$ is given by:
\begin{align*}
A(H_n)= &
\begin{bmatrix}
A(G) & 0 \\
0 & A(G)+I_n
\end{bmatrix}   
\end{align*}
The characteristic polynomial of above matrix is given by:
\begin{align*}
\phi(H_n:x)= &
\begin{vmatrix}
xI-A(G) & 0 \\
0 & xI-(A(G)+I_n)
\end{vmatrix}   
\end{align*}
It follows that if $\lambda_1,\lambda_2,\cdots,\lambda_n$ are eigenvalues of $A$ then,
\begin{align*}
\phi(H_n:x) & =  \prod_{i=1}^{n} (x-\lambda_i)(x-(\lambda_i+1) 
\end{align*}
The roots of above characteristic polynomial are:
$$x=\lambda_i,x=\lambda_i+1$$, for each $i=1,2,\cdots,n$ \\
Here, 
\begin{align}
E(H_n) & =  \sum_{i=1}^{n} \left( \left| \lambda_i-\frac{n}{2n} \right|+\left| \lambda_i+1-\frac{n}{2n} \right| \right) \nonumber \label{eq1} \\
 & =  \sum_{i=1}^{n} \left( \left| \lambda_i-\frac{1}{2} \right|+\left| \lambda_i+\frac{1}{2} \right| \right)  
\end{align}
Suppose that $\left| \lambda_i \right| \geq \frac{1}{2}$, for all $1 \leq i \leq n$ then 
$$
\left| \lambda_i - \frac{1}{2}\right|=
\left\{\begin{array}{ll}
\left| \lambda_i\right| - \dfrac{1}{2},\, &\text{if}\; \lambda_i \ge 0\\\\
\left| \lambda_i\right| + \dfrac{1}{2},\, &\text{if}\; \lambda_i < 0
\end{array}
\right. 
$$
and
$$
\left| \lambda_i + \frac{1}{2}\right|=
\left\{\begin{array}{ll}
\left| \lambda_i\right| + \dfrac{1}{2},\, \text{if}\; \lambda_i \ge 0\\ \\
\left| \lambda_i\right| - \dfrac{1}{2},\, \text{if}\; \lambda_i < 0
\end{array}
\right.
$$
Therefore, from equation \ref{eq1}
\begin{align*}
E(H_n) & =  \sum_{i=1}^{n} \left( \left| \lambda_i-\frac{1}{2} \right|+\left| \lambda_i+\frac{1}{2} \right| \right) \\
& = \sum_{\lambda_i\ge 0} \left( \left| \lambda_i-\frac{1}{2} \right|+\left| \lambda_i+\frac{1}{2} \right| \right)+\sum_{\lambda_i< 0} \left( \left| \lambda_i-\frac{1}{2} \right|+\left| \lambda_i+\frac{1}{2} \right| \right) \\
& = \sum_{\lambda_i\ge 0} \left(  |\lambda_i|-\frac{1}{2}+|\lambda_i|+\frac{1}{2}  \right)+\sum_{\lambda_i< 0} \left( |\lambda_i|+\frac{1}{2}+|\lambda_i|-\frac{1}{2}  \right) \\
 & = 2\left(\sum_{\lambda_i\ge 0}   |\lambda_i| +\sum_{\lambda_i< 0}  |\lambda_i|  \right) \\
 & = 2 \sum_{i=1}^n |\lambda_i| \\
 & =2 E(G)
\end{align*}
\begin{example} 
We now given an example of graph $G$ such that $E(G)=E(G_\sigma)$ and $0<\sigma<n$. Consider the graph $H=K_3 \cup K_3$ and $H_3=K_3 \cup K_3^l$. The graph $H_3$ contains $6$ vertices and three loops. It is known fact that $E(K_3)=4$ and hence $E(H)=E(K_3 \cup K_3)=2E(G)=2(4)=8$.  
\end{example} 
\begin{figure}[h]
\begin{center}
\includegraphics[scale=2]{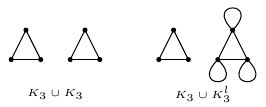}
\end{center}
\end{figure}
The adjacency matrix of $H_3$ is:
\begin{center}
A($H_3$)=
$\begin{bmatrix}
0&1&1&0&0&0\\
1&0&1&0&0&0\\
1&1&0&0&0&0\\
0&0&0&1&1&1\\
0&0&0&1&1&1\\
0&0&0&1&1&1\\
\end{bmatrix}$
\end{center}
The eigenvalues of $K_3^l$ are $3^1,2^1,(-1)^2$ and $0^2$.  \\
Hence,
\begin{center}
$E(H_3)=\left|3-\dfrac{3}{6}\right|+\left|2-\dfrac{3}{6}\right|+2 \left|-1-\dfrac{3}{6}\right|+2\left|0-\dfrac{3}{6}\right|=8$.
\end{center}
Therefore, $E(H)=E(H_3)$.
\begin{theorem}
Let $G$ be the simple graph of order $n$ with eigenvalues $\lambda_1,\lambda_2,\cdots \\,  \lambda_n$ and $G^l$ be the graph obtained from $G$ by adding a loop on each vertex of $G$. Let $p$ and $q$ be non-negative integer and $p+q=m$ then $E((pG \cup qG^l)_{qn})=mE(G)$, if $|\lambda_i| \ge max\left(\frac{p}{m},\frac{q}{m}\right)$, for each $i=1,2,\cdots,n$. 
\end{theorem}
\noindent \textbf{\emph{Proof}}: Let $H_{qn}=pG \cup qG^l$. The graph $H_{qn}$ contains $mn$ vertices and $qn$ loops.  The adjacency matrix of $H_{qn}$ is given by:

\begin{tiny}
\begin{align*}
A(H_{qn})= &
\begin{bmatrix}
A(G) & 0  & \cdots & 0 & 0  & 0 & \cdots & 0 \\
0 & A(G)  & \cdots & 0 & 0  & 0 & \cdots & 0 \\
\vdots & \vdots  & \cdots & \vdots &  \vdots  & \vdots &  \cdots & \vdots \\
0 & 0  & \cdots & A(G) & 0  & 0 & \cdots & 0 \\
0 & 0  & \cdots & 0 & A(G)+I_n  & 0 & \cdots & 0 \\
0 & 0  & \cdots & 0 & 0  & A(G)+I_n & \cdots & 0 \\
\vdots & \vdots  & \cdots & \vdots &  \vdots  & \vdots &  \cdots & \vdots \\
0 & 0  & \cdots & 0 & 0  & 0 & \cdots & A(G)+I_n \\
\end{bmatrix}   
\end{align*}
\end{tiny}
The characteristic polynomial of above matrix is given by:
\begin{tiny}
\begin{align*}
\phi(H_{qn}:x)= &
\begin{vmatrix}
xI-A(G)  & \cdots & 0 & 0  &  \cdots & 0 \\
\vdots   & \ddots & \vdots &  \vdots  & \vdots &  0  \\
0 & \cdots & xI-A(G) & 0  & \cdots  & 0 \\
0 & \cdots  & 0  & xI-(A(G)+I_n)   & \cdots & 0 \\
\vdots & \vdots  & \vdots & \vdots &  \ddots  & \vdots  \\
0 & \cdots  & 0 & 0 & \cdots   & xI-(A(G)+I_n) \\
\end{vmatrix}  
\end{align*}
\end{tiny}
It follows that if $\lambda_1,\lambda_2,\cdots,\lambda_n$ are eigenvalues of $A$ then,
\begin{align*}
\phi(H_{qn}:x) & =  \prod_{i=1}^{n}(x-\lambda_i)^p(x-(\lambda_i+1))^q 
\end{align*}
The roots of above characteristic polynomial are:
$$x=\lambda_i(p-times),\,x=\lambda_i+1(q-times)$$, for each $i=1,2,\cdots,n$ \\
Here, 
\begin{align}
E(H_{qn}) & =  \sum_{i=1}^{n} \left(p \left| \lambda_i-\frac{qn}{mn} \right|+q\left| \lambda_i+1-\frac{qn}{mn} \right| \right) \nonumber  \\
 & =  \sum_{i=1}^{n} \left( p\left| \lambda_i-\frac{q}{m} \right|+q\left| \lambda_i+\frac{m-q}{m} \right| \right)\nonumber \\
 & =  \sum_{i=1}^{n} \left( p\left| \lambda_i-\frac{q}{m} \right|+q\left| \lambda_i+\frac{p}{m} \right| \right)     \label{eq2}
\end{align}
$\bf Case-i:$ Either $p > q$ \textbf{or} $p < q$ \\[2mm]
$\Rightarrow$ $max\left(\dfrac{p}{m},\dfrac{q}{m}\right)=\dfrac{p}{m}$ \textbf{or} $max\left(\dfrac{p}{m},\dfrac{q}{m}\right)=\dfrac{q}{m}$\\[2mm]
If $max\left(\dfrac{p}{m},\dfrac{q}{m}\right)=\dfrac{p}{m}$ then we suppose $\left| \lambda_i \right| \geq \dfrac{p}{m}>\dfrac{q}{m}$ \textbf{and}\\ [2mm]if $max\left(\dfrac{p}{m},\dfrac{q}{m}\right)=\dfrac{q}{m}$ then we suppose $\left| \lambda_i \right| \geq \dfrac{q}{m}>\dfrac{p}{m}$, for all $1 \leq i \leq n$.\\ [2mm]
Therefore,   
$$
\left| \lambda_i - \frac{q}{m}\right|= 
\left\{\begin{array}{ll}
\left| \lambda_i\right| - \dfrac{q}{m},\, &\text{if}\; \lambda_i \geq 0\\\\
\left| \lambda_i\right| + \dfrac{q}{m},\, &\text{if}\; \lambda_i < 0
\end{array}
\right.  
$$
and
$$
\left| \lambda_i + \frac{p}{m}\right|= 
\left\{\begin{array}{ll}
\left| \lambda_i\right| + \dfrac{p}{m},\, &\text{if}\; \lambda_i \ge 0\\\\
\left|\lambda_i\right| - \dfrac{p}{m},\, &\text{if}\; \lambda_i < 0
\end{array}
\right.
$$
Hence, from equation \ref{eq2}
\begin{align*}
E(H_{qn}) & =  \sum_{i=1}^{n} \left( p\left| \lambda_i-\frac{q}{m} \right|+q\left| \lambda_i+\frac{p}{m} \right| \right) \\
& = \sum_{\lambda_i\ge 0} \left( p\left| \lambda_i-\frac{q}{m} \right|+q\left| \lambda_i+\frac{p}{m} \right| \right)+\sum_{\lambda_i< 0} \left( p\left| \lambda_i-\frac{q}{m} \right|+q\left| \lambda_i+\frac{p}{m} \right| \right) \\
& = \sum_{\lambda_i\ge 0} \left(  p|\lambda_i|-\frac{pq}{m}+q|\lambda_i|+\frac{pq}{m}  \right)+\sum_{\lambda_i< 0} \left( p|\lambda_i|+\frac{pq}{m}+q|\lambda_i|-\frac{pq}{m}  \right) \\
 & = p\left(\sum_{\lambda_i\ge 0}   |\lambda_i| +\sum_{\lambda_i< 0}  |\lambda_i|  \right)+ q\left(\sum_{\lambda_i\ge 0}   |\lambda_i| +\sum_{\lambda_i< 0}  |\lambda_i|  \right) \\
 & = p \sum_{i=1}^n |\lambda_i|+q \sum_{i=1}^n |\lambda_i| \\
 & = p E(G)+qE(G) \\
 & =(p+q) E(G)\\
 & =m E(G)
\end{align*}
$\bf Case-ii$ If $p=q$ then we assume $\left| \lambda_i \right| \geq \frac{p}{m}$, for all $1 \leq i \leq n$ then 
$$
\left| \lambda_i - \frac{p}{m}\right|=
\left\{\begin{array}{ll}
\left| \lambda_i\right| - \dfrac{p}{m},\, &\text{if} \;\lambda_i \ge 0\\\\
\left| \lambda_i\right| + \dfrac{p}{m},\, &\text{if} \; \lambda_i < 0
\end{array}
\right.  
$$
and 
$$
\left| \lambda_i + \frac{p}{m}\right|= 
\left\{\begin{array}{ll}
\left| \lambda_i\right| + \dfrac{p}{m},\, &\text{if} \;\lambda_i \ge 0\\\\
\left| \lambda_i\right| - \dfrac{p}{m},\, &\text{if} \; \lambda_i < 0
\end{array}
\right.
$$
Therefore, from equation \ref{eq2}
\begin{align*}
E(H_{qn}) & =  \sum_{i=1}^{n} \left( p\left| \lambda_i-\frac{q}{m} \right|+p\left| \lambda_i+\frac{p}{m} \right| \right) \\
& = p \left[\sum_{\lambda_i\ge 0} \left( \left| \lambda_i-\frac{p}{m} \right|+\left| \lambda_i+\frac{p}{m} \right| \right)+\sum_{\lambda_i< 0} \left( \left| \lambda_i-\frac{p}{m} \right|+\left| \lambda_i+\frac{p}{m} \right| \right) \right] \\
& = p \left[\sum_{\lambda_i\ge 0} \left(  |\lambda_i|-\frac{p}{m}+|\lambda_i|+\frac{p}{m}  \right)+\sum_{\lambda_i< 0} \left( |\lambda_i|+\frac{p}{m}+|\lambda_i|-\frac{p}{m}  \right) \right] \\
 & = p\left(\sum_{\lambda_i\ge 0}   2|\lambda_i| +\sum_{\lambda_i< 0}  2|\lambda_i|  \right)\\
 & = 2pE(G)\\
 & =(p+q) E(G)\\
 & =m E(G)
\end{align*}
\section{Concluding Remarks}
In existing literature, the energy of graphs with self loops is defined. The energy of simple graph $E(G)$ and the energy of graph with self loop $E(G_\sigma)$ are same if $\sigma=0 \ \text{or} \ n$. In the present work we have obtained graphs such that $E(G_\sigma)=E(G)$ and $0<\sigma<n$.
%\acknowledgment{Acknowledgment (if there is any) should be placed here.}
% References should be ordered alphabetically.
\singlespacing

\end{document}